\documentclass[12pt,showkeys]{amsart}
\usepackage{color}
\usepackage{graphicx}
\usepackage{subfigure}
\usepackage{amssymb}
\usepackage{amsmath}
\usepackage{multirow}
\usepackage{enumerate}
\usepackage{epstopdf}
\usepackage{cite,stmaryrd,txfonts}
\usepackage{tikz}

\input{undertilde}
\usetikzlibrary{snakes}

\usepackage{verbatim}

\usepackage{epic}

\usepackage{empheq}

\def\cequiv{\raisebox{-1.5mm}{$\;\stackrel{\raisebox{-3.9mm}{=}}{{\sim}}\;$}}

\def\r{\mathbf{r}}
\def\s{\mathbf{s}}
\def\dx{\mathrm{dx}}

\def\dS{\,\mathrm{ds}\,}

\def\uu{\undertilde{u}}
\def\uw{\undertilde{w}}
\def\uv{\undertilde{v}}
\def\uV{\undertilde{V}}
\def\uH{\undertilde{H}}

\def\uf{\undertilde{f}}

\def\curl{{\rm curl}}
\def\dv{{\rm div}}

\def\ruvps{\mathring{\uV}{}^{\rm PS}}

\def\vwmh{V^{\rm RMW}_{h0}}

\newtheorem{theorem}{Theorem}
\newtheorem{remark}[theorem]{Remark}

\newtheorem{lemma}[theorem]{Lemma}

\newcounter{mnote}
\let\oldmarginpar\marginpar
\renewcommand\marginpar[1]{\-\oldmarginpar[\raggedleft\footnotesize #1]%
  {\raggedright\footnotesize #1}}

\begin{document}

\title[Minimal consistent element for biharmonic equation]{Minimal consistent finite element space for the biharmonic equation on quadrilateral grids}

\author{Shuo Zhang}
\thanks{The author is supported partially by the National Natural Science Foundation of China with Grant No. 11471026 and National Centre for Mathematics and Interdisciplinary Sciences, Chinese Academy of Sciences.}
\address{LSEC, Institute of Computational Mathematics and Scientific/Engineering Computing, Academy of Mathematics and System Sciences, Chinese Academy of Sciences, Beijing 100190, People's Republic of China}
\email{szhang@lsec.cc.ac.cn}



\begin{abstract}
In this paper, a finite element space is presented on quadrilateral grids which can provide consistent discretization for the biharmonic equations.  The space consists of piecewise quadratic polynomials and is of minimal degree for the variational problem. 
\end{abstract}

\maketitle

\section{introduction}
\label{sec:intr}

In the study of qualitative and numerical analysis of partial differential equations and, in general, of approximation theory, we are often interested in the approximation of functions in Sobolev spaces by piecewise polynomials defined on a partition of the domain. In order for simpler interior structure, lower degree polynomials are often expected to be used. It is of theoretical and practical interests whether and how a minimal-degree finite element space can be constructed for certain problems, namely, particularly, whether and how a consistent finite element space can be constructed for $H^m$ elliptic problem with $m$-th degree polynomials.  

When the subdivision consists of simplexes, a systematic family minimal-degree family of nonconforming finite elements has been proposed by Wang and Xu \cite{Wang.M;Xu.J2013} for $2m$-th order elliptic partial differential equations in $R^n$ for any $n\geqslant m$ with polynomials with degree $m$. Known as Wang-Xu or Morley-Wang-Xu family, the elements have been playing bigger and bigger role in numerical analysis. The elements are constructed based on the perfect matching between the dimension of $m$-th degree polynomials and the dimension of $(n-k)$-subsimplexes with $1\leqslant k\leqslant m$. The generalization to the cases $n<m$ is attracting more and more research interests, c.f., e.g., \cite{Wu.S;Xu.J2017}. 

When the subdivision consists of geometrical shapes other than simplex, the problem is more complicated. The minimal conforming element spaces have been constructed for $H^m$ problem on $\mathbb{R}^n$ rectangle grids by Hu and Zhang \cite{Hu.J;Zhang.Sy2015}, where $Q_k$ polynomials are used for $2k$-th order problems. Some low-degree rectangle elements have been designed, such as the rectangular Morley element and incomplete $P_3$ element for biharmonic equation. It remains open if the degrees can be further reduced for consistent nonconforming element spaces with minimal degree. Generally, the cells of shapes as simple as rectangle can share interfaces with more neighbour cells, and more continuity restrictions will need a higher degree of polynomials, which is generally higher than the order of the underlying Sobolev space. This way, it seems extremely difficult, if not impossible, to construct consistent finite elements in the formulation of Ciarlet's triple (c.f. \cite{Ciarlet.P1978,Wang.M;Shi.Z2013mono}) with $m$-th degree for $H^m$ problems on even rectangular grids.  When approaches are not restricted on Ciarlet's triple, consistent finite element space with linear polynomials can be constructed for Poisson equation on quadrilateral grids via a more global approach\cite{Park.C;Sheen.D2003,Hu.J;Shi.Z2005}; for higher-order elliptic equations, though, the minimal-degree construction of finite elements on even rectangular grids remains open, and let alone general quadrilateral subdivisions.

In this paper, we study the minimal-degree finite element construction for biharmonic equation on quadrilateral grids. We will present a finite element space that consists of piecewise quadratic polynomials on quadrilateral grids, and which can provide consistent discretization for the biharmonis equation. 

The finite element functions on quadrilateral cells are not constructed by the usual bilinear mapping approach. Actually, as pointed in \cite{Arnold.D;Boffi.D;Falk.R;Gastaldi.L2001,Arnold.D;Boffi.D;Falk.R2002}, 
for quadrilateral finite elements that are constructed by starting from a given finite dimensional space of polynomials on the unit reference square, a necessary and sufficient condition for certain approximation accuracy is that the space on reference cell contains the space of $Q_r$ type, namely of all polynomial functions of degree r separately in each variable. This will stops us from constructing a consistent finite element space with polynomial space of $P_r$ type. The way we construct finite element space is similar to the way discussed in \cite{Park.C;Sheen.D2003,Hu.J;Shi.Z2005,Park.C;Sheen.D2013}, where local polynomial function spaces are used on every cell directly. The space constructed in the present paper can be viewed as a reduced quadrilateral Morley element space. This observation admits us to present the consistency estimation in an easy way. Similar to elements in, e.g., \cite{Fortin.M;Soulie.M1983} and \cite{Park.C;Sheen.D2003}, and many spline type methodss, the continuity restriction of the finite element function is more than determining a local polynomial.  and the standard scaling argument could not be used directly. We instead figure out the exact connection between the finite element functions designed in this paper and the element in \cite{Park.C;Sheen.D2003}, and construct the approximation estimate in an indirect way. The exact connection can be viewed as a specialization of the one given in \cite{Zhang.S2016nm}. Moreover, this relation indicates an analogue correspondence with the exactness relation between the Crouzeix-Raviart element and the Morley element on triangular grids.

The remaining of the paper is organized as follows. In Section \ref{sec:pre}, some preliminaries are collected. In Section \ref{sec:quadele}, two finite elements are defined on quadrilateral grids. They are each the quadrilater analogue of the rotated $Q_1$ element and the Wilson element on rectangle grids, and they are useful auxiliary element in this paper. In Sections \ref{sec:min} and \ref{sec:impl}, a minimal-degree finite element scheme for biharmonic equation is designed, and its convergence analysis and implementation are presented. Finally, in Section \ref{sec:conc}, some concluding remarks are given.

\section{Preliminaries}
\label{sec:pre}

In this paper, we use these notation. Let $\Omega\subset\mathbb{R}^2$ be a simply-connected Lipschitz domain with a boundary of piecewise segment, and $\Gamma=\partial\Omega$ be the boundary, with $\mathbf{n}$ the outward unit normal vector. Denote by $H^1(\Omega)$, $H^1_0(\Omega)$, $H^2(\Omega)$, and $H^2_0(\Omega)$ the standard Sobolev spaces as usual, and $L^2_0(\Omega):=\{w\in L^2(\Omega):\int_\Omega w\dx=0\}$.  Denote $\undertilde{H}{}^1_0(\Omega):=(H^1_0(\Omega))^2$, $\undertilde{L}^2(\Omega)=(L^2(\Omega))^2$; we use $``\undertilde{\cdot}"$ for vector valued quantities in the present paper. We use usual symbols, and use the subscript $\cdot_h$ for a cell by cell operation when a grid is involved. Denote $\mathring\uH{}^1_0(\Omega):=\{\uw\in \uH{}^1_0(\Omega):\dv\uw=0\}$.

\subsection{Geometry of convex quadrilateral grid}

Let $Q$ be a convex quadrilateral with $a_i$ the vertices and $e_i$ the edges, $i=1:4$. See Figure \ref{fig:convquad} for an illustration. Let $m_i$ be the mid-point of $e_i$, then the quadrilateral $\square m_1m_2m_3m_4$ is a parallelogram(\cite{Park.C;Sheen.D2003}). The cross point of $m_1m_3$ and $m_2m_4$, which is labelled as $O$, is the midpoint of both $m_1m_3$ and $m_2m_4$. Denote $\mathbf{r}=\overrightarrow{Om_3}$ and $\mathbf{s}=\overrightarrow{Om_4}$. Then the coordinates of the vertices in the coordinate system $\r O\s$ are $a_1(-1-\alpha,1-\beta)$, $a_2(-1+\alpha,-1+\beta)$, $a_3(1-\alpha,-1-\beta)$ and $a_4(1+\alpha,1+\beta)$ for some $\alpha,\beta$. 
Since $Q$ is convex, $|\alpha|+|\beta|<1$. Without loss of generality, we assume $\alpha>0$, $\beta>0$ and $\r\times\s>0$.

Define the shape regularity indicator of the cell $Q$ by $\mathcal{R}_Q:=\max\{\frac{|\mathbf{r}||\mathbf{s}|}{\mathbf{r}\times\mathbf{s}},\frac{|\mathbf{r}|}{|\mathbf{s}|},\frac{|\mathbf{s}|}{|\mathbf{r}|}\}$. Evidently $\mathcal{R}_Q\geqslant 1$, and $\mathcal{R}_Q=1$ if and only if $Q$ is a square. A given family of quadrilateral triangulations $\{\mathcal{G}_h\}$ of $\Omega$ is said to be regular, if all the shape regularity indicators of the cells of all the triangulations are uniformly bounded.

\begin{figure}[htbp]
\begin{tikzpicture}
\draw (0.05,0.26)node{$O$};

\draw[dashed](-2,2)--(3,2);
\draw[dashed](-3.5,-2)--(1.5,-2);
\draw[dashed](-3.5,-2)--(-2,2);
\draw[dashed](1.5,-2)--(3,2);

\draw[dashed](0.5,2)--(-1,-2);
\draw[dashed](2.25,0)--(-2.75,0);

\draw(-2.345,0.275)node{\Large$m_1$};
\draw(-0.75,-1.9)node{\Large$m_2$};
\draw(2.1,0.35)node{\Large$m_3$};
\draw(0.85,1.75)node{\Large$m_4$};

\draw(-3.4,0.175)node{\Large$e_1$};
\draw(-0.95,-2.4)node{\Large$e_2$};
\draw(2.7,0.05)node{\Large$e_3$};
\draw(0.4,2.35)node{\Large$e_4$};

\draw[->,line width=3pt,-latex](-0.25,0)-- node[auto] {$\r$} (2.25,0);
\draw[->,line width=3pt,-latex](-0.25,0)-- node[auto] {$\s$} (0.5,2);

\draw(3.725,2.6)--(-2.725,1.4);
\draw(3.725,2.6)--(0.775,-2.6);
\draw(-2.775,-1.4)--(0.775,-2.6);
\draw(-2.775,-1.4)-- (-2.725,1.4);

\draw(-2.9,1.7)node{\Large$a_1$};
\draw(-2.8,-1.7)node{\Large$a_2$};
\draw(0.9,-2.85)node{\Large$a_3$};
\draw(3.85,2.8)node{\Large$a_4$};

\draw[dashed](0.5,2)--(2.25,0);
\draw[dashed](2.25,0)--(-1,-2);
\draw[dashed](-1,-2)--(-2.75,0);
\draw[dashed](-2.75,0)--(0.5,2);

\end{tikzpicture}

\caption{Illustration of a convex quadrilateral $Q$.}\label{fig:convquad}
\end{figure}
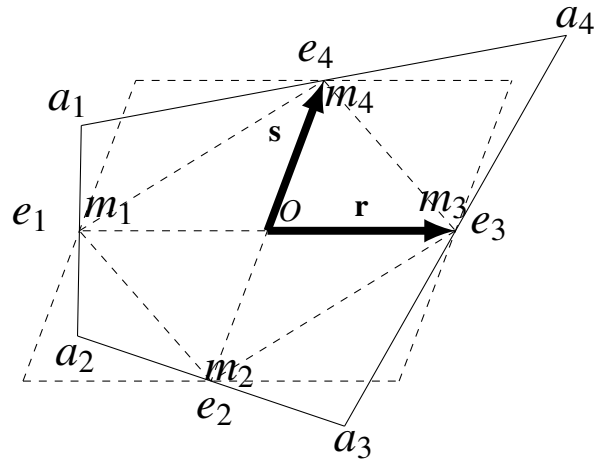

Define two linear functions $\xi$ and $\eta$ by $\xi(a\r+b\s)=a\ \ \mbox{and}\ \ \eta(a\r+b\s)=b.$ The two functions play the same role on quadrilateral as that of barycentric coordinate on triangles. 

Technically, we construct a table (Table \ref{tab:tecasis} below, c.f. \cite{Zhang.S2016nm}) about the evaluation of some functions firstly. Here and after, while ``$\int$" is the usual integral symbol, we use ``$\fint$" for the spatial average on the integral domain. This table will be useful in both theoretical analysis and practical programming. 
\begin{table}[htbp]
\begin{tabular}{c|cccccccccc}
 function($u$) & 1 & $\eta$ & $\xi$ & $\eta^2$ & $\xi^2$ &$\eta^2-\xi^2$
 \\
\hline $\displaystyle\fint_{e_1}u\dS$ & 1 & 0 & -1 & $\displaystyle \frac{(1-\beta)^2}{3}$ & $\displaystyle 1+\frac{\alpha^2}{3}$&$ \displaystyle 1+\frac{\alpha^2}{3} -\frac{(1-\beta)^2}{3}$
\\
\hline $\displaystyle\fint_{e_2}u\dS$ & 1 & -1 & 0 & $\displaystyle 1+\frac{\beta^2}{3}$ & $\displaystyle \frac{(1-\alpha)^2}{3}$ & $\displaystyle \frac{(1-\alpha)^2}{3}- 1-\frac{\beta^2}{3}$
\\
\hline $\displaystyle\fint_{e_3}u\dS$ & 1 & 0 & 1 & $\displaystyle \frac{(1+\beta)^2}{3}$ & $\displaystyle 1+\frac{\alpha^2}{3}$ & $\displaystyle 1+\frac{\alpha^2}{3}-\frac{(1+\beta)^2}{3}$
\\
\hline $\displaystyle\fint_{e_4}u\dS$ & 1 & 1 & 0 & $\displaystyle 1+\frac{\beta^2}{3}$ & $\displaystyle \frac{(1+\alpha)^2}{3}$ & $\displaystyle \frac{(1+\alpha)^2}{3}-1-\frac{\beta^2}{3}$
\\
\hline
\end{tabular}
\caption{Boundary average of some functions.}\label{tab:tecasis}
\end{table}

\subsection{Subdivisions and finite element spaces}

Let $\mathcal{G}_h$ be a regular subdivision of domain $\Omega$, with the cells being convex quadrilaterals; i.e., $\displaystyle\Omega=\cup_{Q\in\mathcal{G}_h}Q$. Let $\mathcal{N}_h$ denote the set of all the vertices, $\mathcal{N}_h=\mathcal{N}_h^i\cup\mathcal{N}_h^b$, with $\mathcal{N}_h^i$ and $\mathcal{N}_h^b$ consisting of the interior vertices and the boundary vertices, respectively. Similarly, let $\mathcal{E}_h=\mathcal{E}_h^i\bigcup\mathcal{E}_h^b$ denote the set of all the edges, with $\mathcal{E}_h^i$ and $\mathcal{E}_h^b$ consisting of the interior edges and the boundary edges, respectively. For an edge $e$, $\mathbf{n}_e$ is a unit vector normal to $e$, and $\tau_e$ is a unit tangential vector of $e$ such that $\mathbf{n}_e\times\tau_e>0$. On the edge $e$, we use $\llbracket\cdot\rrbracket_e$ for the jump across $e$. 

Denote by $\mathfrak{F}$ the number of cells of the triangulation; denote by $\mathfrak{X}$, $\mathfrak{X}_I$, $\mathfrak{X}_B$ and $\mathfrak{X}_C$ the number of vertices, internal vertices, boundary vertices, and corner vertices, respectively; and denote by $\mathfrak{E}$, $\mathfrak{E}_I$ and $\mathfrak{E}_B$ the number of edges, internal edges, and boundary edges, respectively. Euler's formula states that $\mathfrak{F}+\mathfrak{X}=\mathfrak{E}+1$.

\subsubsection{Quadrilateral Morley element space}
The quadrilateral Morley element (\cite{Park.C;Sheen.D2013}) is defined by $(Q,P_Q^{\rm QM},D_Q^{\rm QM})$ with
\begin{enumerate}
\item $Q$ is a convex quadrilateral;
\item $P_Q^{\rm QM}=P_2(Q)+{\rm span}\{\xi^3,\eta^3\}$;
\item the components of $D_Q^{\rm QM}=\{d_i^{\rm QM},d_{i+4}^{\rm QM}\}_{i=1:4}$ for any $v\in H^2(Q)$ are: 
$$
d^{\rm QM}_i(v)=v(a_i),\ a_i\ \mbox{the\ vertices\ of}\ T;\ \ d^{\rm QM}_{i+4}(v)=\fint_{e_i}\partial_{\mathbf{n}_{e_i}}v\dS,\ e_i\ \mbox{the\ edges\ of}\ T.
$$
\end{enumerate}

Given a quadrilateral grid $\mathcal{G}_h$ of $\Omega$, define the quadrilateral Morley element space $V^{\rm QM}_h$ as
\begin{multline*}
\qquad V^{\rm QM}_h:=\{w_h\in L^2(\Omega):w_h|_Q\in P_Q^{\rm QM},\ \forall\,Q\in \mathcal{G}_h,\ w_h(a)\ \mbox{is\ continuous\ at}\ a\in\mathcal{N}_h,\  \\ 
\fint_e\partial_{\mathbf{n}_e}w_h\dS\ \mbox{is\ continuous\ across}\ e\in\mathcal{E}_h^i\}.\qquad
\end{multline*}
And, associated with $H^2_0(\Omega)$, define 
$$
V^{\rm QM}_{h0}:=\{w_h\in V^{\rm QM}_h:w_h(a)\ \mbox{vanishes\ at}\ a\in\mathcal{N}_h^b,\ \fint_e\partial_{\mathbf{n}_e}w_h\dS\ \mbox{vanishes\ at}\ e\in\mathcal{E}_h^b\}.
$$

\subsubsection{Quadrilateral Lin-Tobiska-Zhou (QLTZ) element space}

The QLTZ element (\cite{Lin.Q;Tobiska.L;Zhou.A2005,Zhang.S2016nm}) is defined by $(Q,P_Q^{\rm QLTZ},D_Q^{\rm QLTZ})$ with
\begin{enumerate}
\item $Q$ is a convex quadrilateral;
\item $P_Q^{\rm QLTZ}=P_1(Q)+{\rm span}\{\xi^2,\eta^2\}$;
\item the components of $D_Q^{\rm QLTZ}=\{d_0^{\rm QLTZ}\}_{i=0:4}$ for any $v\in H^1(Q)$ are: 
$$
d_0^{\rm QLTZ}(v)=\fint_Qv\dx,\ \ \mbox{and}\ \ d_i^{\rm QLTZ}(v)=\fint_{e_i}v\dS,\ e_i\ \mbox{the\ edges\ of}\ T,\ i=1:4.
$$
\end{enumerate}

Associated with $H^1(\Omega)$, define a finite element space $V_h^{\rm QLTZ}$ by
$$
V_h^{\rm QLTZ}:=\{w\in L^2(\Omega):w|_Q\in P_Q^{\rm QLTZ},\ \forall\,Q\in \mathcal{G}_h,\ \fint_e w\dS\ \mbox{is\ continuous\ at}\ e\in\mathcal{E}_h^i\},
$$
and associated with $H^1_0(\Omega)$, define a finite element space $V_{h0}$ by
$$
V_{h0}^{\rm QLTZ}:=\{w_h\in V_h^{\rm QLTZ}:\fint_ew_h\dS=0\ \mbox{at\ }e\in \mathcal{E}_h^b\}.
$$

\begin{lemma}\label{lem:esqmltz}
(Lemma 8 of \cite{Zhang.S2016nm})
$\curl_h V^{\rm QM}_{h0}=\{\uw{}_h\in \uV{}^{\rm QLTZ}_{h0}:\dv_h\uw{}_h=0\}:=\mathring{\uV}{}^{\rm QLTZ}_{h0}$. 
\end{lemma}

\subsubsection{Park-Sheen element space}
The Park-Sheen element space (\cite{Park.C;Sheen.D2003}) is a piecewise linear nonconforming finite element space for $H^1$ problem. It is defined as
$$
V_h^{\rm PS}:=\{w\in L^2(\Omega):w|_Q\in P_1(Q),\ \forall\,Q\in \mathcal{G}_h,\ \fint_e w\dS\ \mbox{is\ continuous\ at}\ e\in\mathcal{E}_h^i\},
$$
and, associated with $H^1_0(\Omega)$, as
$$
V_{h0}^{\rm PS}:=\{w_h\in V_h^{\rm PS}:\fint_ew_h\dS=0\ \mbox{at\ }e\in \mathcal{E}_h^b\}.
$$
Similarly, denote $\mathring{\uV}{}^{\rm PS}_{h0}:=\{\uv{}_h\in\uV{}^{\rm PS}_{h0}:\dv_h\uv{}_h=0\}$.

\begin{lemma}(Theorem 4.2 of \cite{Altmann.R;Carstensen.C2012}, also \cite{Park.C;Sheen.D2003,Hu.J;Shi.Z2005})
There is a constant $C$ depending on the regularity of $\mathcal{G}_h$ only, such that given $w\in H^2(\Omega)$, 
$$
\inf_{v_h\in V^{\rm PS}_h}(\|w-v_h\|_{0,\Omega}^2+\sum_{Q\in\mathcal{G}_h}h_Q^2\|\nabla (w-v_h)\|_{0,Q}^2)
\leqslant C\sum_{Q\in\mathcal{G}_h}h_Q^4\|w\|_{2,Q}^2.
$$
\end{lemma}

\subsection{Piecewise linear element space on criss-cross grids}

Given $\mathcal{T}_h$ a triangulation of $\Omega$, denote by $V^{\rm le}_h$ the linear element space on $\mathcal{T}_h$, $V^{\rm le}_{h0}=V^{\rm le}_h\cap H^1_0(\Omega)$, and $\uV{}^{\rm le}_{h0}=(V_{h0})^2$. Denote by $\mathbb{P}_{h0}$ the space of piecewise constant the integral of which is zero. Consider the Stokes problem: find $(\uu,p)\in \uH{}^1_0(\Omega)\times L^2_0(\Omega)$ such that 
\begin{equation}\label{eq:stokes}
\left\{
\begin{array}{lll}
(\nabla\uu,\nabla \uv)+(p,\dv\uv)&=(\uf,\uv)&\forall\,\uv\in\uH{}^1_0(\Omega)
\\
(\dv\uu,q)&=0&\forall\,q\in L^2_0(\Omega),
\end{array}
\right.
\end{equation}
and its discretization: find $(\uu{}_h,p_h)\in \uV{}^{\rm le}_{h0}\times \mathbb{P}_{h0}$, such that
\begin{equation}\label{eq:p1p0}
\left\{
\begin{array}{lll}
(\nabla\uu{}_h,\nabla \uv{}_h)+(p_h,\dv\uv{}_h)&=(\uf,\uv{}_h)&\forall\,\uv{}_h\in\uV{}^{\rm le}_{h0}
\\
(\dv\uu{}_h,q_h)&=0&\forall\,q_h\in \mathbb{P}_{h0}.
\end{array}
\right.
\end{equation}
In general, \eqref{eq:p1p0} does not provide a stable discretization of \eqref{eq:stokes}, but it can provide good approximation of $\uu$ for some special cases. We adopt the hypothesis below.

\paragraph{\bf Hypothesis G} (c.f. \cite{Pitkaranta.J;Stenberg.R1985,Qin.J;Zhang.S2007})  $\mathcal{G}_h$ is generated by uniformly refining a shape-regular quadrilateral subdivision $\mathcal{G}_{4h}$ of $\Omega$ twice.

\begin{figure}[htbp]
\begin{tikzpicture}

\draw(-4.275,2.6)--(-10.725,1.4);
\draw(-4.275,2.6)--(-7.225,-2.6);
\draw(-10.775,-1.4)--(-7.225,-2.6);
\draw(-10.775,-1.4)-- (-10.725,1.4);

\draw(3.725,2.6)--(-2.725,1.4);
\draw(3.725,2.6)--(0.775,-2.6);
\draw(-2.775,-1.4)--(0.775,-2.6);
\draw(-2.775,-1.4)-- (-2.725,1.4);

\draw[dashed](3.725,2.6)--(-2.775,-1.4);
\draw[dashed](-2.725,1.4)--(0.775,-2.6);

\draw(0.5,2)--(-1.153,-0.39);
\draw(2.25,0)--(-1.153,-0.39);
\draw(-1.153,-0.39)--(-1,-2);
\draw(-1.153,-0.39)--(-2.75,0);

\draw(-1.153,-0.39)node{.};

\draw[dashed](0.5,2)--(2.25,0);
\draw[dashed](2.25,0)--(-1,-2);
\draw[dashed](-2.75,0)--(-1,-2);
\draw[dashed](0.5,2)--(-2.75,0);

\draw(-1.94,0.495)--(-1.1125,1.7);
\draw(-1.94,0.495)--(-2.7375,0.7);
\draw(-1.94,0.495)--(-0.3265,0.805);
\draw(-1.94,0.495)--(-1.9515,-0.195);

\draw(-1.95,-0.89)--(-1.9515,-0.195);
\draw(-1.95,-0.89)--(-2.7625,-0.7);
\draw(-1.95,-0.89)--(-1.0765,-1.195);
\draw(-1.95,-0.89)--(-1.8875,-1.7);

\draw(-0.181,-1.535)--(0.5485,-0.195);
\draw(-0.181,-1.535)--(-1.0765,-1.195);
\draw(-0.181,-1.535)--(1.48625,-1.3);
\draw(-0.181,-1.535)--(-0.13625,-2.3);

\draw(1.29,1.08)--(2.1125,2.3);
\draw(1.29,1.08)--(-0.3265,0.805);
\draw(1.29,1.08)--(2.9675,1.3);
\draw(1.29,1.08)--(0.5485,-0.195);


\draw(-2.9,1.7)node{\Large$a_1$};
\draw(-2.8,-1.7)node{\Large$a_2$};
\draw(0.9,-2.85)node{\Large$a_3$};
\draw(3.85,2.8)node{\Large$a_4$};

\draw(-4,0)node{\Large $\Longrightarrow$};

\draw(-10.9,1.7)node{\Large$a_1$};
\draw(-10.8,-1.7)node{\Large$a_2$};
\draw(-7.1,-2.85)node{\Large$a_3$};
\draw(-4.15,2.8)node{\Large$a_4$};

\end{tikzpicture}

\caption{Illustration of uniformly refining twice a quadrilateral. A quadrilateral is refined by connecting the intersect point of the diagonals to the mid points of the edges. c.f. \cite{Qin.J;Zhang.S2007}.}\label{fig:refinequad}
\end{figure}
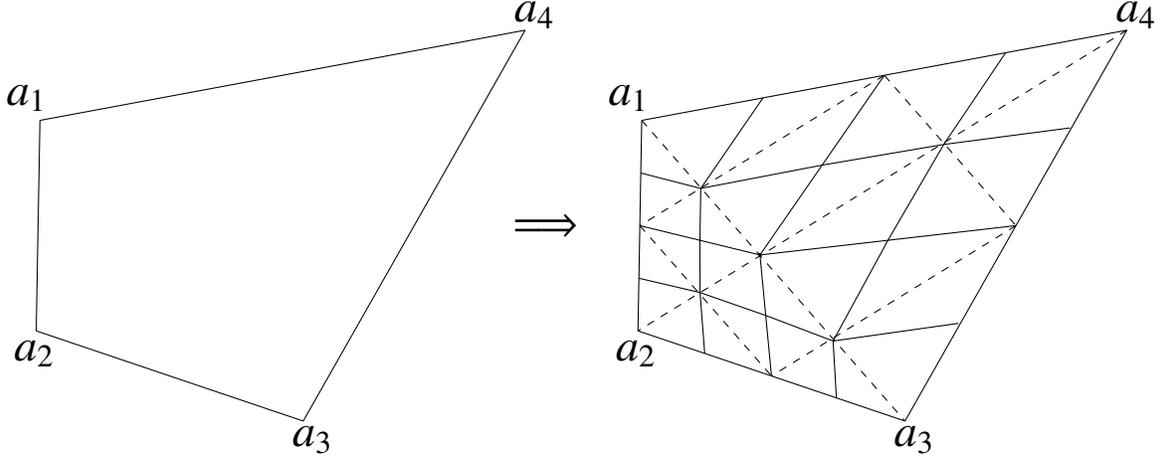

\begin{lemma}\label{lem:ledf}(Theorem 4.2 of \cite{Qin.J;Zhang.S2007})
Let $\mathcal{G}_h$ be a quadrilateral subdivision of $\Omega$ that satisfies {\bf Hypothesis G}. Let $\mathcal{T}_h$ be a triangulation of $\Omega$ made by a criss-cross refinement of $\mathcal{G}_h$, namely dividing each quadrilateral into four subtriangles by the two diagonals. Let $(\uu,p)$ and $(\uu{}_h,p_h)$ be the solutions of \eqref{eq:stokes} and \eqref{eq:p1p0}, respectively. Then
\begin{equation}
\|\uu-\uu{}_h\|_{1,\Omega}\leqslant Ch\|\uu\|_{2,\Omega}
\end{equation}
provided $(\uu,p)\in\uH{}^2(\Omega)\times H^1(\Omega)$.
\end{lemma}

\begin{remark}
Denote $\mathring{\uV}{}^{\rm le}_{h0}:=\{\uv{}_h\in \uV{}^{\rm le}_{h0}:\dv\uv{}_h=0\}$. 
Lemma \ref{lem:ledf} reveals that the 
\begin{equation}\displaystyle
\inf_{\uv{}_h\in \mathring{\uV}{}^{\rm le}_{h0}}\|\uu-\uv\|_{1,h}\leqslant Ch\|\uu\|_{2,\Omega},\ \mbox{for}\,\uu\in \mathring\uH{}^1_0(\Omega)\cap \uH{}^2(\Omega).
\end{equation}
\end{remark}
Several finite elements are mentioned in this paper, and for readers' convenience, we list the abbreviations below. 
\begin{description}
\item[QM] Quadrilateral Morley
\item[QLTZ] Quadrilateral Lin-Tobiska-Zhou
\item[PS] Park-Sheen
\item[le] linear element
\item[RQM] Reduced Quadrilateral Morley
\item[QGB] Quadrilateral Generalized Bilinear (see Section \ref{sec:qgb})
\item[QW] Quadrilateral Wilson (see Section \ref{sec:qw})
\end{description}

\section{Two quadrilateral finite elements}
\label{sec:quadele}

In this section, we present two new finite elements on quadrilateral grids. 

\subsection{A quadrilateral generalized bilinear element}
\label{sec:qgb}

Define a quadrilateral generalized bilinear element by $(Q,P_Q^{\rm QGB},D_Q^{\rm QGB})$, with
\begin{enumerate}
\item $Q$ is a convex quadrilateral;
\item $P_Q^{\rm QGB}=P_1(Q)+{\rm span}\{\xi^2-\eta^2\}$;
\item the components of $D_Q^{\rm QGB}=\{d_i^{\rm QGB}\}_{i=1:4}$ for any $v\in H^1(Q)$ are: 
$$
d^{\rm QGB}_i(v)=\fint_{e_i}v,\ e_i\ \mbox{the\ edges\ of}\ Q.
$$
\end{enumerate}
Define on a quadrilateral $Q$
\begin{eqnarray}
\phi_1=\frac{3}{8}(\xi^2-\eta^2)+\frac{\beta-2}{4}\xi+\frac{\alpha-4}{4}\eta+\frac{2-\alpha^2+\beta^2}{8}
\\
\phi_2=-\frac{3}{8}(\xi^2-\eta^2)-\frac{\beta}{4}\xi+\frac{2-\alpha}{4}\eta+\frac{\alpha^2-\beta^2+2}{8}
\\
\phi_3=\frac{3}{8}(\xi^2-\eta^2)+\frac{\beta+2}{4}\xi+\frac{\alpha+4}{4}\eta+\frac{2-\alpha^2+\beta^2}{8}
\\
\phi_4=-\frac{3}{8}(\xi^2-\eta^2)-\frac{\beta}{4}\xi-\frac{2+\alpha}{4}\eta+\frac{\alpha^2-\beta^2+2}{8},
\end{eqnarray}
then $\phi_i\in P^{\rm QGB}_Q$ and $d^{\rm QGB}_i(\phi_j)=\delta_{ij}$, $i,j=1:4$. This proves the uni-solvence. Given $v\in P^{\rm QGB}_Q$, then $v=\sum_{i=1}^4 d_i^{\rm QGB}(v)\phi_i.$ Direct calculation leads to that
$$
v=\frac{3}{8}(d_1^{\rm QGB}(v)+d_3^{\rm QGB}(v)-d_2^{\rm QGB}(v)-d_4^{\rm QGB}(v))(\xi^2-\eta^2)+v',\ \ \mbox{with}\ v'\in P_1(Q).
$$ 
This way, given $v\in P^{\rm QGB}_{Q}$, $v\in P_1(Q)$ if and only if $d_1^{\rm QGB}(v)+d_3^{\rm QGB}(v)=d_2^{\rm QGB}(v)+d_4^{\rm QGB}(v)$.

Note that the nodal parameters make sense for $v\in H^1(Q)$. Define the interpolator ${\rm I}^{\rm QGB}_Q:H^1(Q)\to P^{\rm QGB}_Q$ by $d^{\rm QGB}_i({\rm I}^{\rm QGB}_Q v)=d^{\rm QGB}_i(v)$ for $i=1:4$. 
\begin{lemma}\label{lem:stabszhang}(c.f. proof of Lemma 3 of \cite{Zhang.S2016nm})
There is a constant $C$, depending on the regularity of the quadrilateral $Q$ only, such that 
\begin{equation}
|\fint_{e_i}w-\fint_{e_j}w|\leqslant C|w|_{1,Q},\ \ \forall\,w\in H^1(Q).
\end{equation}
\end{lemma}
\begin{lemma}
The interpolation ${\rm I}^{\rm QGB}_Q$ is stable in the sense that $\|\nabla {\rm I}_h^{\rm QGB}w\|_{0,Q}\leqslant C\|\nabla w\|_{0,Q}$, with a generic constant $C$ depending on the regularity of the quadrilateral only.
\end{lemma}
\begin{proof}
For simplicity, denote $d_i:=\fint_{e_i}w$, and $\phi=\sum_{i=1:4}d_i\phi_i$. Then
$$
\partial_{\r}\phi=3\xi/4(d_1-d_2+d_3-d_4)+\beta/4(d_1-d_2+d_3-d_4)+1/2(d_3-d_1),
$$
and
$$
\partial_{\s}\phi=3/4\eta(-d_1+d_2-d_3+d_4)+\alpha/4(d_1-d_2+d_3-d_4)+(d_3-d_1)+1/2(d_2-d_4).
$$
We have used the fact that $\partial_{\bf r}\eta=\partial_{\bf s}\xi=0$. This way, by Lemma \ref{lem:stabszhang},
\begin{multline*}
\int_Q|\nabla \phi|^2\cequiv \fint_Q|\partial_{\r}\phi|^2+|\partial_{\s}\phi|^2
\leqslant C\left((d_1-d_2)^2+(d_3-d_4)^2+(d_1-d_3)^2+(d_2-d_4)^2\right) \leqslant C|w|_{1,Q}^2.
\end{multline*}
The proof is completed. 
\end{proof}

Given a quadrilateral subdivision $\mathcal{G}_h$, define the generalized bilinear element space by 
$$
V^{\rm QGB}_h:=\{w\in L^2(\Omega):w|_Q\in P^{\rm QGB}_Q,\ \forall\,Q\in\mathcal{G}_h, \ \fint_ew\ \mbox{is\ continuous\ across}\ e\in\mathcal{E}^i_h\},
$$ 
and, associated with $H^1_0(\Omega)$, 
$$
V^{\rm QGB}_{h0}:=\{w\in V^{\rm QGB}_h:\ \fint_ew=0\ \mbox{at}\ e\in\mathcal{E}^b_h\}.
$$
Define ${\rm I}^{\rm QGB}_h:H^1(\Omega)\to V^{\rm QGB}_h$ by $({\rm I}^{\rm QGB}_hv)|_Q={\rm I}^{\rm QGB}_Q(v|_Q)$ for any $Q\in\mathcal{G}_h$. Then ${\rm I}^{\rm QGB}_hH^1_0(\Omega)\subset V^{\rm QGB}_{h0}$. Denote $\undertilde{\rm I}{}^{\rm QGB}_h=({\rm I}^{\rm QGB}_h)^2$.

\begin{lemma}\label{lem:propqgb}
\begin{enumerate}
\item $|{\rm I}_h^{\rm QGB}v-v|_{1,h}\leqslant Ch|v|_{2,\Omega}$ for $v\in H^1_0(\Omega)\cap H^2(\Omega)$;
\item let $\mathcal{T}_h$ be a triangulation of $\Omega$ made by a criss-cross refinement of $\mathcal{G}_h$; then 
$$
{\rm I}^{\rm QGB}_hV^{\rm le}_{h0}(\mathcal{T}_h) = V^{\rm PS}_{h0}(\mathcal{G}_h),\ \mbox{and}\ \ \undertilde{\rm I}{}^{\rm QGB}_h\mathring\uV{}^{\rm le}_{h0}(\mathcal{T}_h)=\mathring{\uV}{}^{\rm PS}_{h0}(\mathcal{G}_h).
$$
\end{enumerate}
\end{lemma}

\begin{proof}
\begin{enumerate}
\item As $V^{\rm PS}_{h0}\subset V^{\rm QGB}_{h0}$, ${\rm I}^{\rm QGB}_h w_h=w_h$ for $w_h\in  V^{\rm QGB}_{h0}$, and 
$$
|{\rm I}^{\rm QGB}_h v-v|_{1,h}=\inf_{w_h\in  V^{\rm QGB}_{h0}}|{\rm I}^{\rm QGB}_h (v-w_h)-(v-w_h)|_{1,h}\leqslant C \inf_{w_h\in V^{\rm PS}_{h0}}|v-w_h|_{1,h}\leqslant Ch|v|_{2,\Omega}.
$$
\item Given $w_h\in V^{\rm le}_{h0}(\mathcal{T}_h)$, then $\fint_e w=\frac{1}{2}(w_h(L_e)+w_h(R_e))$ for every $e\in\mathcal{E}_h$, where $L_e$ and $R_e$ are the two ends of $e$. Now given $Q\in\mathcal{G}_h$ with edges $e^Q_i$, $i=1:4$ with anticlockwise order, evidently $\fint_{e^Q_1}w_h+\fint_{e^Q_3}w_h=\fint_{e^Q_2}w_h+\fint_{e^Q_4}w_h$, and $ \fint_{e^Q_1}{\rm I}^{\rm QGB}_h w_h+\fint_{e^Q_3}{\rm I}^{\rm QGB}_h w_h=\fint_{e^Q_2} {\rm I}^{\rm QGB}_h w_h+\fint_{e^Q_4} {\rm I}^{\rm QGB}_h w_h$, thus ${\rm I}^{\rm QGB}_hw_h|_Q\in P_1(Q)$. Namely ${\rm I}^{\rm QGB}_hw_h|_Q\in \mathrm{V}^{\rm PS}_{h0}(\mathcal{G}_h)$. Similarly, $\undertilde{\rm I}{}^{\rm QGB}_h\uV{}^{\rm le}_{h0}(\mathcal{T}_h)=\uV{}^{\rm PS}_{h0}(\mathcal{G}_h)$. It is easy to verify that $\int_K \dv\undertilde{\rm I}{}^{\rm QGB}_h \uw{}_h=\int_K \uw{}_h$, and the assertion follows. 
\end{enumerate}
The proof is completed. 
\end{proof}
\begin{remark}
A similar result to Item 2 of Lemma \ref{lem:propqgb} can be found in \cite{Hu.J;Shi.Z2005}. 
\end{remark}

\subsection{A quadrilateral Wilson element}
\label{sec:qw}

The quadrilateral Wilson element is defined by $(Q,P_Q^{\rm QW},D_Q^{\rm QW})$, with
\begin{enumerate}
\item $Q$ is a convex quadrilateral;
\item $P_Q^{\rm QW}=P_2(Q)$;
\item the components of $D_Q^{\rm QW}=\{d_i^{\rm QW}\}_{i=1:6}$ for any $v\in H^2(Q)$ are: 
$$
d^{\rm QW}_i(v)=v(a_i),\ a_i\ \mbox{the\ vertices\ of}\ T; \ \ d^{\rm QW}_5(v)=\int_Q\partial_{\xi\xi}v,\ d^{\rm QW}_6(v)=\int_Q\partial_{\eta\eta}v.
$$
\end{enumerate}

\begin{lemma}
The quadrilateral Wilson element is uni-solvent. 
\end{lemma}
\begin{proof}
Define

\begin{eqnarray}
\phi_4=-\frac{1}{8(\beta+1)}\left[-(\xi+\eta+1)^2+(\alpha+\beta-1)^2\right]
\\
\phi_3=\frac{1}{4(1-\alpha)}((\xi+1)^2-\alpha^2)-\frac{1+\alpha}{1-\alpha}\phi_4
\\
\phi_2=\frac{1}{2(\beta-1)}(\eta-(1-\beta))-\frac{1}{1-\beta}\phi_3-\frac{\beta}{\beta-1}\phi_4
\\
\phi_1=1-\phi_2-\phi_3-\phi_4
\\
\phi_5=(\alpha-\beta+1)\xi^2+2\alpha\xi+2\alpha\xi\eta+(\alpha+\beta-1)(1-\alpha^2)
\\
\phi_6=\xi^2-\eta^2+2\beta\xi-2\alpha\eta+\beta^2-\alpha^2.
\end{eqnarray}
It is easy to verify that $d^{\rm QW}_i(\phi_j)=\delta_{ij}$. The proof is completed. 
\end{proof}
Given a quadrilateral subdivision $\mathcal{G}_h$, define the quadrilateral Wilson element space by 
$$
V^{\rm QW}_h:=\{w\in L^2(\Omega):w|_Q\in P_2(Q),\ \forall\,Q\in\mathcal{G}_h, \ w\ \mbox{is\ continuous\ at}\ a\in\mathcal{N}_h\},
$$ 
and, associated with $H^1_0(\Omega)$, 
$$
V^{\rm QW}_{h0}:=\{w\in V^{\rm QW}_h:\ w(a)=0\ \mbox{at}\ a\in\mathcal{N}^b_h\}.
$$

\section{A minimal consistent finite element space for biharmonic equation}
\label{sec:min}

In this section, we study the discretization of biharmonic equation on quadrilateral grids, and present a consistent finite element method with piecewise quadratic polynomials. Consider the model problem: 
\begin{equation}
\left\{
\begin{array}{ll}
\Delta^2u=f&\mbox{in}\ \Omega
\\
u=\frac{\partial u}{\partial\mathbf{n}}=0&\mbox{on}\ \partial\Omega.
\end{array}
\right.
\end{equation}
The variational problem is, with $f\in L^2(\Omega)$,
\begin{equation}\label{eq:model}
\mbox{find}\ u\in H^2_0(\Omega),\ \mbox{such\ that}\ (\nabla^2u,\nabla^2v)=(f,v),\ \forall\,v\in H^2_0(\Omega). 
\end{equation}
Let $\mathcal{G}_h$ be a quadrilateral subdivision of $\Omega$ and define piecewise quadratic element spaces on $\mathcal{G}_h$ by
\begin{itemize}
\item reduced quadrilateral Morley (RQM for short) finite element space
\begin{multline}
\qquad V^{\rm RQM}_h:=\{v_h\in L^2(\Omega):v_h|_Q\in P_2(Q),\ \forall\,Q\in\mathcal{G}_h,\ v_h(a)\ \mbox{is\ continuous\ on}\ a\in\mathcal{X}_h,
\\
\mbox{and}\ \int_e\frac{\partial v_h}{\partial n_e}\ \mbox{is\ continuous\ along}\ e\in\mathcal{E}^i_h\};\qquad
\end{multline}
\item homogeneous RQM finite element space:
\begin{equation}
V^{\rm RQM}_{h0}:=\{v_h\in V^{\rm RQM}_{h}:v_h(a)=0\ \mbox{on}\ a\in \mathcal{X}_h\setminus\mathcal{X}_h^i,\ \int_e\frac{\partial v_h}{\partial n_e}=0\ \mbox{on}\ e\in\mathcal{E}_h\setminus\mathcal{E}^i_h\}.
\end{equation}
\end{itemize}
Consider the finite element problem for \eqref{eq:model}: 
\begin{equation}\label{eq:modeldis}
\mbox{find}\ u_h\in V^{\rm RQM}_{h0},\ \mbox{such\ that}\ \sum_{K\in\mathcal{G}_h}(\nabla^2u_h,\nabla^2v)=(f,v),\ \forall\,v\in V^{\rm RQM}_{h0}. 
\end{equation}
As $V^{\rm RQM}_{h0}\subset V^{\rm QM}_{h0}$, The well-posed-ness of \eqref{eq:modeldis} follows by \cite{Park.C;Sheen.D2013}.

The main result of this paper is the theorem below.
\begin{theorem}\label{thm:main}
Let $u$ and $u_h$ be the solutions of \eqref{eq:model} and \eqref{eq:modeldis}, respectively. If $u\in H^3(\Omega)$, then
\begin{equation} 
|u-u_h|_{2,h}\leqslant Ch(|u|_{3,\Omega}+h\|f\|_{0,\Omega}).
\end{equation}
\end{theorem}

We postpone the proof of Theorem \ref{thm:main} after some technical lemmas. 

\subsection{Approximation of $V^{\rm RQM}_{h0}$}

\subsubsection{Approximation of Park-Sheen space revisited}

Recall $\ruvps_{h0}:=\{\uw{}_h\in (V^{\rm PS}_{h0})^2:\dv\uw{}_h|_K=0,\ \ \mbox{on}\, K\in \mathcal{G}_h\}.$  We are going to construct the approximation result below. 
\begin{lemma}\label{lem:approxvfps}
Provided {\bf Hypothesis G} for $\mathcal{G}_h$, and let $\uw\in\mathring{\uH}{}^1_0(\Omega)\cap \uH{}^2(\Omega)$. Then
\begin{equation}
\inf_{\uv{}_h\in\ruvps_{h0}}\|\uw-\uv{}_h\|_{1,h}\leqslant Ch\|\uw\|_{2,\Omega}. 
\end{equation}
\end{lemma}

\begin{proof}
Let $\mathcal{T}_h$ be a triangulation of $\Omega$ made by a criss-cross refinement of $\mathcal{G}_h$. Define $\undertilde{\rm P}{}^{\rm le}_{h0}:\mathring{\uH}{}^1_0(\Omega)\to \mathring \uV{}^{\rm le}_{h0}$ such that 
\begin{equation}
(\nabla \undertilde{\rm P}{}^{\rm le}_{h0}\uu,\nabla \uv{}_h)=(\nabla \uu,\nabla\uv{}_h),\quad \uu\in \mathring{\uH}{}^1_0(\Omega),\ \forall\,\uv{}_h\in \mathring\uV{}^{\rm le}_{h0}.
\end{equation}
By Lemma \ref{lem:ledf}, the operator  $\undertilde{\rm P}{}^{\rm le}_{h0}$ is well defined, and 
\begin{equation}
\|\uu-\undertilde{\rm P}{}^{\rm le}_{h0}\uu\|_{1,\Omega}\leqslant Ch\|\uu\|_{2,\Omega},\ \ \mbox{for}\ \uu\in \mathring{\uH}{}^1_0(\Omega)\cap \undertilde{H}^2(\Omega).
\end{equation}
Therefore, by Lemma \ref{lem:propqgb}, given $\uu\in\mathring{\uH}{}^1_0(\Omega)\cap \uH{}^2(\Omega)$,
\begin{multline*}
\inf_{\uv{}_h\in\ruvps_{h0}}|\uu-\uv{}_h|_{1,h}
\leqslant |\uu-\undertilde{\rm I}{}^{\rm QGB}_{h}\undertilde{\rm P}{}^{\rm le}_{h0}\uu|_{1,\Omega}
\leqslant |\uu-\undertilde{\rm I}\,{}^{\rm QGB}_{h}\uu|_{1,\Omega}+ |\undertilde{\rm I}\,{}^{\rm QGB}_{h}\uu-\undertilde{\rm I}\,{}^{\rm QGB}_{h}\undertilde{\rm P}{}^{\rm le}_{h0}\uu|_{1,\Omega}
\\
\leqslant |\uu-\undertilde{\rm I}\,{}^{\rm QGB}_{h}\uu|_{1,\Omega}+ C|\uu-\undertilde{\rm P}{}^{\rm le}_{h0}\uu|_{1,\Omega}\leqslant Ch|\uu|_{2,\Omega}.
\end{multline*}
This completes the proof. 
\end{proof}

\subsubsection{Approximation of RQM space}

\begin{lemma}
$\curl V^{\rm RQM}_{h0}=\mathring{\uV}{}^{\rm PS}_{h0}$.
\end{lemma}
\begin{proof}
Evidently, $\curl V^{\rm RQM}_{h0}\subset \mathring{\uV}{}^{\rm PS}_{h0}$. On the other hand, given $\uw{}_h\in \mathring{\uV}{}^{\rm PS}_{h0}\subset \mathring{\uV}{}^{\rm QLTZ}_{h0}$, by Lemma \ref{lem:esqmltz}, there exists a $\varphi{}_h\in V^{\rm QM}_{h0}$, such that $\curl_h\varphi{}_h=\uw{}_h$. Note that $\curl\,(\varphi_h|_K)\in (P_1(K))^2$ for any $K$, and thus $\varphi_h|_K\in P^2(K)$ on any $K\subset\mathcal{G}_h$, and we obtain $\varphi_h\in V^{\rm RQM}_{h0}$ and thus $\curl V^{\rm RQM}_{h0}\supset \mathring \uV{}^{\rm PS}_{h0}$. The proof is completed. 
\end{proof}

\begin{theorem}\label{thm:appRQM}
Provided {\bf Hypothesis G} for $\mathcal{G}_h$, and let $u\in H^2_0(\Omega)\cap H^3(\Omega)$. Then
\begin{equation}\label{eq:enae}
\inf_{v_h\in V^{\rm RQM}_{h0}}|u-v{}_h|_{2,h}\leqslant Ch|u|_{3,\Omega}. 
\end{equation}
\end{theorem}
\begin{proof}
Given $u\in H^2_0(\Omega)\cap H^3(\Omega)$, $\curl u\in \mathring{\uH}{}^1_0(\Omega)\cap \uH^2(\Omega)$. Thus
$$
\inf_{v_h\in V^{\rm RQM}_{h0}}|u-v{}_h|_{2,h} = \inf_{v_h\in V^{\rm RQM}_{h0}}|\curl u-\curl_hv{}_h|_{1,h} =\inf_{\uv{}_h\in\ruvps_{h0}}|\curl u-\uv{}_h|_{1,h}\leqslant Ch|\curl u|_{2,\Omega}= Ch|u|_{3,\Omega}. 
$$
This completes the proof.  
\end{proof}

\subsection{Proof of Theorem \ref{thm:main}}
The proof of  Theorem \ref{thm:main} is similar to the analysis in \cite{Shi.Z1990} and \cite{Park.C;Sheen.D2013}. By the second Strang Lemma, we have
\begin{equation}
\|u-u_h\|_{2,h}\leqslant C\bigg(\inf\limits_{v_h\in V^{\rm RQM}_{h0}}\|u-u_h\|_{2,h}+\sup\limits_{0\neq w_h\in V^{\rm RQM}_{h0}}\frac{|(\nabla^2u,\nabla_h^2w_h)-(f,w_h)|}{\|w_h\|_{2,h}}\bigg),
\end{equation}
where the first term is the approximation error and the second one is the consistency error. 

Let $\mathcal{T}_h$ be a triangulation of $\Omega$ generated by a criss-cross refinement of $\mathcal{G}_h$, and $V^{\rm le}_{h0}(\mathcal{T}_h)$ be the homogeneous linear element space on $\mathcal{T}_h$. Denote by $I_h$ the nodal interpolation of $V^{\rm le}_{h0}$. Then by Green formula,
\begin{equation}
(f,I_h w_h)=(\Delta^2 u,I_h w_h)=-\int_\Omega \nabla\Delta u\cdot\nabla I_h w_h. \label{eq9}
\end{equation}

The integration by parts yields
\begin{equation}
\begin{split}
(\nabla^2u,\nabla^2_hu_h)-(f,w_h)& =-\sum_{K\in \mathcal{T}_h}\int_K\nabla\Delta u\cdot\nabla(w_h-I_h w_h)-\sum_{K\in \mathcal{T}_h}\int_Kf(w_h-I_h w_h)
\\
&+\sum_{K\in \mathcal{T}_h}\int_{\partial K}\frac{\partial^2 u}{\partial n^2}\frac{\partial w_h}{\partial n}\;\mathrm{d}s+\sum_{K\in \mathcal{T}_h}\int_{\partial K}\frac{\partial ^2u}{\partial s\partial n}\frac{\partial w_h}{\partial s}\;\mathrm{d}s.\label{eq11}
\end{split}
\end{equation}
where $\frac{\partial}{\partial s}$ and $\frac{\partial}{\partial n}$ are tangential and normal derivatives along element boundaries, respectively.
The Cauchy-Schwarz inequality and the standard interpolation error estimate  lead to
\begin{equation}
\bigg|\sum_{K\in\mathcal{T}_h}\int_Kf(w_h-I_h w_h)\;\mathrm{d}x_1\mathrm{d}x_2\bigg|\leq Ch^2||f||_{L^2(\Omega)}|w_h|_h. \label{eq14}
\end{equation}
and
$$
(f,I_h w_h-w_h)\leqslant Ch^2\|f\|_{0,\Omega}\|w_h\|_{2,h}.
$$

Again, as $V^{\rm RQM}_{h0}\subset V^{\rm QM}_{h0}$, it holds by  \cite{Park.C;Sheen.D2013} that 
\begin{equation}
\sum_{K\in \mathcal{T}_h}\int_{\partial K}\frac{\partial^2 u}{\partial n^2}\frac{\partial w_h}{\partial n}\;\mathrm{d}s+\sum_{K\in \mathcal{T}_h}\int_{\partial K}\frac{\partial ^2u}{\partial s\partial n}\frac{\partial w_h}{\partial s}\;\mathrm{d}s\leqslant Ch\|u\|_{3,\Omega}|w_h|_{2,h}.
\end{equation}
and finally
\begin{equation}
|(\nabla^2u,\nabla^2_hv_h)-(f,v_h)|\leqslant Ch|v_h|_{2,h}(|u|_{3,\Omega}+h\|f\|_{0,\Omega}),\ \ \forall\,v_h\in V^{\rm RQM}_{h0}.
\end{equation}
For approximation error, we refer to Theorem \ref{thm:appRQM}. The proof is completed. 
\begin{remark}
Based on the exact relation between $V^{\rm RQM}_{h0}$ and $\mathring{V}^{\rm PS}_{h0}$, the consistency error estimate can also be established based on the techniques in \cite{Altmann.R;Carstensen.C2012}.
\end{remark}

%

%
%
%
\section{On the implementation of the finite element scheme}
\label{sec:impl}

It is evident that, the restrictions of the continuity of the RQM element function across internal edges are more than necessary to shape a quadratic polynomial on a quadrilateral. For special cases, such as for a rectangular grid on a rectangle domain, a linearly independent set of basis functions of the space can be given, while in general, the RQM element space may not be easy to be constructed by figuring out local basis functions. In this section, we present an explicit description of the basis functions on rectangle grids, and present an alternative approach how \eqref{eq:modeldis} can be implemented for general quadrilateral subdivision. 

\subsection{Local basis functions of the RQM element space on rectangular subdivisions}
\label{sec:hsofRQM} 

In this part, we consider the case that the domain can be covered by a rectangular subdivision. Namely, let $\Omega\subset\mathbb{R}^2$ be a rectangle. For a subdivision of any domain $\omega$, again, we use $\mathcal{E}_h$, $\mathcal{E}_h^i$, $\mathcal{X}_h$ and $\mathcal{X}_h^i$ for the set of faces, interior faces, edges, interior edges, vertices and interior vertices, respectively. For any edge $e\in\mathcal{E}_h$, denote by $\undertilde{t}{}_e$ the unit tangential vector along $e$. Particularly, if none of the vertices of a cell $K$ is on the boundary of $\omega$, we name this cell an {\it interior} cell. We use $\mathcal{K}^i_h$ for the set of interior cells. We use the symbol $``\#"$ for the cardinal of a set. Let $\mathcal{G}_h$ be a shape regular rectangle subdivision of $\Omega$.

\begin{lemma}
Let $\omega$ be a rectangle, and $\mathcal{T}_\omega$ be a $3\times 3$ subdivision of $\omega$.  Let $V^{\rm RQM}_{h0}$ be the homogeneous RQM finite element space defined on $\mathcal{T}_\omega$. Then $\dim(V^{\rm RQM}_{h0})=1$.
\end{lemma}
\begin{proof}
We begin with the local construction of a quadratic polynomial versus a rectangle. Let $K$ be a rectangle with vertices $a_i$ and edges $\Gamma_i$, c.f. Figure \ref{fig:illus}, left. 
\begin{figure}[htbp]

\setlength{\unitlength}{1mm}

\begin{picture}(60,75)(-32,-5)

\thicklines\color{black}

\put(-60,0){\line(1,0){40}}
\put(-60,30){\line(1,0){40}}
\put(-60,0){\line(0,1){30}}
\put(-20,0){\line(0,1){30}}

\put(-60.8,-0.8){$\bullet$}
\put(-60.8,29.2){$\bullet$}
\put(-20.8,-0.8){$\bullet$}
\put(-20.8,29.2){$\bullet$}

\put(-59,1){$a_2$}
\put(-59,31){$a_1$}
\put(-19,1){$a_3$}
\put(-19,31){$a_4$}

\put(-60.3,15){$\fint\partial_x$}
\put(-20.3,15){$\fint\partial_x$}
\put(-40,0){$\fint\partial_y$}
\put(-40,30){$\fint\partial_y$}

\put(-59,10){$\Gamma_1$}
\put(-19,10){$\Gamma_3$}
\put(-30,1){$\Gamma_2$}
\put(-30,31){$\Gamma_4$}

\put(0,0){\line(1,0){60}}
\put(0,20){\line(1,0){60}}
\put(0,40){\line(1,0){60}}
\put(0,60){\line(1,0){60}}

\put(0,0){\line(0,1){60}}
\put(20,0){\line(0,1){60}}
\put(40,0){\line(0,1){60}}
\put(60,0){\line(0,1){60}}

\put(-6,10){$\gamma$}
\put(-6,30){$\beta$}
\put(-6,50){$\alpha$}
\put(10,63){$\kappa$}
\put(30,63){$\sigma$}
\put(50,63){$\delta$}

\put(19,8.5){$\bullet$}
\put(19,28.5){$\bullet$}
\put(19,48.5){$\bullet$}
\put(39,8.5){$\bullet$}
\put(39,28.5){$\bullet$}
\put(39,48.5){$\bullet$}

\put(22,10){$Z^3_1$}
\put(22,30){$Z^2_1$}
\put(22,50){$Z^1_1$}
\put(42,10){$Z^3_2$}
\put(42,30){$Z^2_2$}
\put(42,50){$Z^1_2$}

\put(9,19){$\bullet$}
\put(9,39){$\bullet$}
\put(29,19){$\bullet$}
\put(29,39){$\bullet$}
\put(49,19){$\bullet$}
\put(49,39){$\bullet$}

\put(12,21){$Y^2_1$}
\put(12,41){$Y^1_1$}
\put(32,21){$Y^2_2$}
\put(32,41){$Y^1_2$}
\put(52,21){$Y^2_3$}
\put(52,41){$Y^1_3$}

\put(19,19){$\bullet$}
\put(19,39){$\bullet$}
\put(39,19){$\bullet$}
\put(39,39){$\bullet$}

\put(21,21){$X^2_1$}
\put(21,41){$X^1_1$}
\put(41,21){$X^2_2$}
\put(41,41){$X^1_2$}

\put(63,9){$\mathcal{T}^3$}
\put(63,29){$\mathcal{T}^2$}
\put(63,49){$\mathcal{T}^1$}
\put(9,-6){$\mathcal{T}_1$}
\put(29,-6){$\mathcal{T}_2$}
\put(49,-6){$\mathcal{T}_3$}

\put(9,9){$K^3_1$}
\put(9,29){$K^2_1$}
\put(9,49){$K^1_1$}
\put(29,9){$K^3_2$}
\put(29,29){$K^2_2$}
\put(29,49){$K^1_2$}
\put(49,9){$K^3_3$}
\put(49,29){$K^2_3$}
\put(49,49){$K^1_3$}

\end{picture}
\vskip0mm
\caption{Left: Illustration of a rectangle $K$, with width $L$ and height $H$. Right:  Illustration of $\omega$ and $\mathcal{T}_\omega$: $X^i_j$ denotes the vertices, $Y^i_j$ and $Z^i_j$ denote the midpoints, and $K^i_j$ denotes the cells. The Greek letters denote the width (length) of the cells in the same column (row).}\label{fig:illus}
\end{figure}
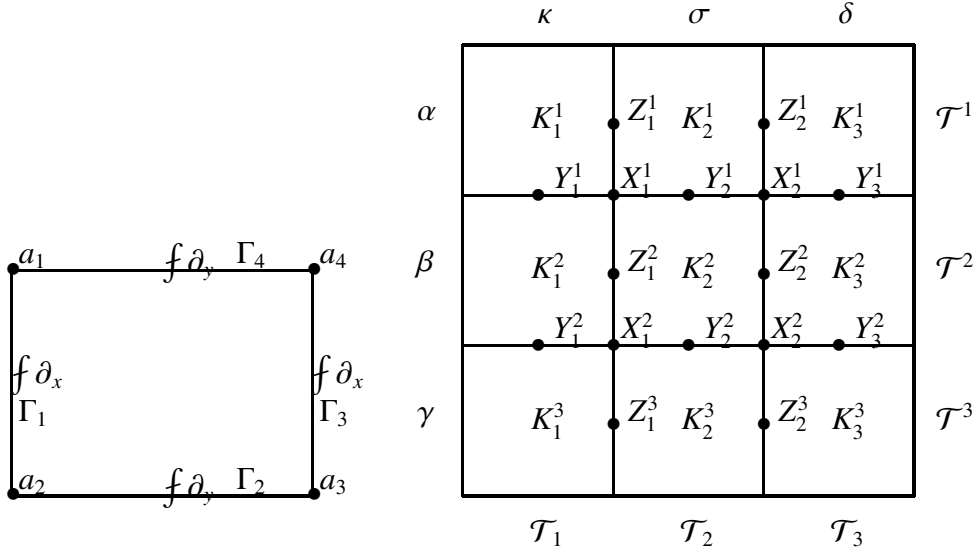
Then, given $\alpha_i,\beta_i\in\mathbb{R}$, $i=1:4$, there exists uniquely a $p\in P_2(K)$, such that 
$$
p(a_i)=\alpha_i,\ \fint_{\Gamma_1}\partial_xp=\beta_1,\ \fint_{\Gamma_2}\partial_yp=\beta_2,\ \fint_{\Gamma_3}\partial_xp=\beta_3,\ \mbox{and}\ \fint_{\Gamma_4}\partial_yp=\beta_4, 
$$ 
if and only if 
\begin{equation}\label{eq:cc}
\frac{\alpha_4-\alpha_1}{L}+\frac{\alpha_3-\alpha_2}{L}=\beta_1+\beta_3,\ \ \  \mbox{and}\ \ \ \frac{\alpha_4-\alpha_3}{H}+\frac{\alpha_1-\alpha_2}{L}=\beta_2+\beta_4.
\end{equation}
This way, under the compatible condition \eqref{eq:cc}, a quadratic polynomial is uniquely determined by its evaluation on vertices and derivative on edges. 

Now let the geometric features of $\omega$ and $\mathcal{T}_\omega$ be labelled as in Figure \ref{fig:illus}, right. Given $\varphi\in V^{\rm RQM}_{h0}(\mathcal{T}_\omega)$, denote by $x^i_j:=\varphi(X^i_j)$, $y^i_j:=\partial_y\varphi(Y^i_j)$ and $z^i_j:=\partial_x\varphi(Z^i_j)$. By the compatible condition \eqref{eq:cc} on every cell, we have, row by row,
\begin{eqnarray}
\frac{x^1_1}{\kappa}=z^1_1,\ -\frac{x^1_1}{\alpha}=y^1_1,\ \frac{x^1_2-x^1_1}{\sigma}=z^1_2+z^1_1,\ -\frac{x^1_2+x^1_1}{\alpha}=y^1_2,\ -\frac{x^1_2}{\alpha}=y^1_3,\ -\frac{x^1_2}{\delta}=z^1_2, 
\\ \frac{x^1_1+x^2_1}{\kappa}=z^2_1,\ \frac{x^1_1-x^2_1}{\beta}=y^1_1+y^2_1,\ \frac{x^1_1+x^1_2-x^2_1-x^2_2}{\beta}=y^1_2+y^2_2,
\\ \frac{x^1_2+x^2_2-x^1_1-x^2_1}{\sigma}=z^2_2+z^2_1, \
-\frac{x^1_2+x^2_2}{\delta}=z^2_2,\ \frac{x^1_2-x^2_2}{\beta}=y^1_3+y^2_3,
\\ \frac{x^2_1}{\kappa}=z^3_1,\ \frac{x^2_1}{\gamma}=y^2_1,\ \frac{x^2_1+x^2_2}{\gamma}=y^2_2,\quad  \frac{x^2_2-x^2_1}{\sigma}=z^3_2+z^3_1,\ -\frac{x^2_2}{\delta}=z^3_2,\quad \frac{x^2_2}{\gamma}=y^2_3.
\end{eqnarray}
We further rewrite the system equivalently to, after adjusting the order,
\begin{eqnarray}
x^1_1-\kappa z^1_1=0,\ x^1_2+\delta z^1_2=0, \ x^1_1+\alpha y^1_1=0,\ x^2_1-\gamma y^2_1=0
\\
x^1_2+\alpha y^1_3=0,\ x^2_2-\gamma y^2_3=0, \ x^2_1-\kappa z^3_1=0,\ x^2_2+\delta z^3_2=0
\end{eqnarray}
\begin{eqnarray}
x^1_1+x^1_2-x^2_1-x^2_2-\beta(y^1_2+y^2_2)=0 \label{eq:38}
\\
x^1_2+x^2_2-x^1_1-x^2_1-\sigma(z^2_2+z^2_1)=0.\label{eq:39}
\end{eqnarray}
\begin{eqnarray}
(1+\frac{\sigma}{\delta})x^1_2-(1+\frac{\sigma}{\kappa})x^1_1=0\label{eq:27}
\\ 
(1+\frac{\beta}{\alpha})x^1_1-(1+\frac{\beta}{\gamma})x^2_1=0\label{eq:29}
\\ 
(1+\frac{\beta}{\alpha})x^1_2-(1+\frac{\beta}{\gamma})x^2_2=0\label{eq:31}
\\
(1+\frac{\sigma}{\delta})x^2_2-(1+\frac{\sigma}{\kappa})x^2_1=0 \label{eq:33}
\end{eqnarray}
\begin{eqnarray}
x^1_2+x^1_1+\alpha y^1_2=0 \label{eq:34}
\\
x^1_1+x^2_1-\kappa z^2_1=0\label{eq:35}
\\
x^1_2+x^2_2+\delta z^2_2=0\label{eq:36}
\\
x^2_1+x^2_2-\gamma y^2_2=0\label{eq:37}
\end{eqnarray}
It is straightforward to verify that $\displaystyle\eqref{eq:38}= -\frac{\beta}{\alpha}\eqref{eq:34}+\frac{\beta}{\gamma}\eqref{eq:37}+\eqref{eq:29}+\eqref{eq:31}$, $\displaystyle\eqref{eq:39}= -\frac{\sigma}{\delta}\eqref{eq:36}+\frac{\sigma}{\kappa}\eqref{eq:35}+\eqref{eq:27}+\eqref{eq:33} $ and $\displaystyle\eqref{eq:31}= \frac{1+\beta/\alpha}{1+\sigma/\delta}\eqref{eq:27}+\frac{1+\sigma/\kappa}{1+\sigma/\delta}\eqref{eq:29}-\frac{1+\beta/\gamma}{1+\sigma/\delta}\eqref{eq:33}.$ Now, eliminate Equations \eqref{eq:31}, \eqref{eq:38} and \eqref{eq:39} from the system, and it is easy to see the remaining fifteen equations are linearly independent and the system admits a one-dimension solution space. This completes the proof. 
\end{proof}
\begin{remark}\label{rem:noicnowm}
If a subdivision $\mathcal{G}_\Omega$ of $\Omega$ does not have any interior cell, then $\dim(\vwmh(\mathcal{G}_\Omega))=0$.
\end{remark}

\begin{theorem}
Let $\mathcal{G}_h$ be a rectangular subdivision of $\Omega$. For the homogeneous RQM element space $\vwmh$ on $\Omega$. Then $\dim(\vwmh)=\#(\mathcal{K}^i_j(\mathcal{T}_h))$, the number of interior cells of the subdivision $\mathcal{T}_h$. 
\end{theorem}
\begin{proof}
We prove the result by showing how a sweeping procedure works. Assume the domain is subdivided by $m\times n$ rectangles. In $y$ direction, let the grid $\mathcal{T}_\omega$ be decomposed to $m$ rows, each being $\mathcal{T}^i$, $1\leqslant i\leqslant m$, and in $x$ direction, let the grid be decomposed to $n$ columns, each being $\mathcal{T}_j$, $1\leqslant j\leqslant n$; see Figure \ref{fig:tran}. Label the vertices by $a_j^i$, $1\leqslant i\leqslant m$, $1\leqslant j\leqslant n$, and the cells by $T^i_j$. That is, $T^i_j=\mathcal{T}^i\cap \mathcal{T}_j$, and the vertices of $T^i_j$ are $a^{i-1}_{j-1}$, $a^{i-1}_{j}$, $a^{i}_{j-1}$ and $a^{i}_{j}$.
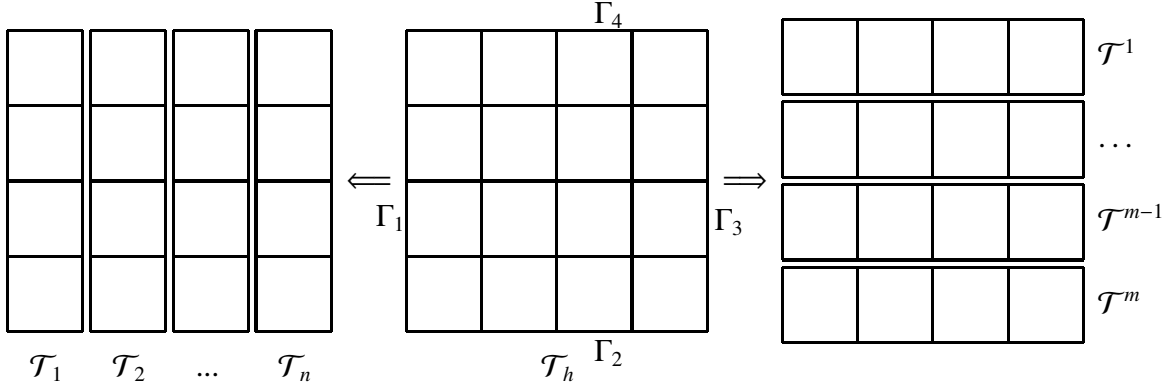
\begin{figure}[htbp]
\setlength{\unitlength}{1mm}

\begin{picture}(150,50)(-15,-5)
\thicklines\color{black}

\put(30,0){\line(0,1){40}}
\put(20,0){\line(0,1){40}}
\put(19,0){\line(0,1){40}}
\put(9,0){\line(0,1){40}}
\put(8,0){\line(0,1){40}}
\put(-2,0){\line(0,1){40}}
\put(-3,0){\line(0,1){40}}
\put(-13,0){\line(0,1){40}}
\put(-10,-6){$\mathcal{T}_1$}
\put(1,-6){$\mathcal{T}_2$}
\put(12,-6){$...$}
\put(23,-6){$\mathcal{T}_n$}

\put(-13,0){\line(1,0){10}}
\put(-2,0){\line(1,0){10}}
\put(9,0){\line(1,0){10}}
\put(20,0){\line(1,0){10}}

\put(-13,10){\line(1,0){10}}
\put(-2,10){\line(1,0){10}}
\put(9,10){\line(1,0){10}}
\put(20,10){\line(1,0){10}}

\put(-13,20){\line(1,0){10}}
\put(-2,20){\line(1,0){10}}
\put(9,20){\line(1,0){10}}
\put(20,20){\line(1,0){10}}

\put(-13,30){\line(1,0){10}}
\put(-2,30){\line(1,0){10}}
\put(9,30){\line(1,0){10}}
\put(20,30){\line(1,0){10}}

\put(-13,40){\line(1,0){10}}
\put(-2,40){\line(1,0){10}}
\put(9,40){\line(1,0){10}}
\put(20,40){\line(1,0){10}}

\put(32,19){$\Longleftarrow$}

\put(36,14){$\Gamma_1$}
\put(81,13.5){$\Gamma_3$}
\put(65,-4){$\Gamma_2$}
\put(65,41){$\Gamma_4$}
\put(40,0){\line(1,0){40}}
\put(40,10){\line(1,0){40}}
\put(40,20){\line(1,0){40}}
\put(40,30){\line(1,0){40}}
\put(40,40){\line(1,0){40}}
\put(40,0){\line(0,1){40}}
\put(50,0){\line(0,1){40}}
\put(60,0){\line(0,1){40}}
\put(70,0){\line(0,1){40}}
\put(80,0){\line(0,1){40}}

\put(58,-6){$\mathcal{T}_h$}

\put(82,19){$\Longrightarrow$}

\put(90,-1.5){\line(1,0){40}}
\put(90,8.5){\line(1,0){40}}
\put(90,9.5){\line(1,0){40}}
\put(90,19.5){\line(1,0){40}}
\put(90,20.5){\line(1,0){40}}
\put(90,30.5){\line(1,0){40}}
\put(90,31.5){\line(1,0){40}}
\put(90,41.5){\line(1,0){40}}

\put(132,2.5){$\mathcal{T}^m$}
\put(132,13.5){$\mathcal{T}^{m-1}$}
\put(132,24.5){$\dots$}
\put(132,35.5){$\mathcal{T}^1$}

\put(90,-1.5){\line(0,1){10}}
\put(90,9.5){\line(0,1){10}}
\put(90,20.5){\line(0,1){10}}
\put(90,31.5){\line(0,1){10}}

\put(90,-1.5){\line(0,1){10}}
\put(90,9.5){\line(0,1){10}}
\put(90,20.5){\line(0,1){10}}
\put(90,31.5){\line(0,1){10}}

\put(100,-1.5){\line(0,1){10}}
\put(100,9.5){\line(0,1){10}}
\put(100,20.5){\line(0,1){10}}
\put(100,31.5){\line(0,1){10}}

\put(110,-1.5){\line(0,1){10}}
\put(110,9.5){\line(0,1){10}}
\put(110,20.5){\line(0,1){10}}
\put(110,31.5){\line(0,1){10}}

\put(120,-1.5){\line(0,1){10}}
\put(120,9.5){\line(0,1){10}}
\put(120,20.5){\line(0,1){10}}
\put(120,31.5){\line(0,1){10}}

\put(130,-1.5){\line(0,1){10}}
\put(130,9.5){\line(0,1){10}}
\put(130,20.5){\line(0,1){10}}
\put(130,31.5){\line(0,1){10}}

\end{picture}

\caption{Illustration of the domain and the triangulation.} \label{fig:tran}
\end{figure}

The interior cells of $\omega$ are $T^i_j$ with $2\leqslant i\leqslant m-1$ and $2\leqslant j\leqslant n-1$. For any interior cell $T^i_j$, there is $3\times 3$ patch, labelled as $P^i_j$, with $T^i_j$ being its center cell. A homogeneous RQM element space $V^{\rm RQM}_{h0}(P^i_j)$ can be constructed on $P^i_j$.  Here, we do not distinct $V^{\rm RQM}_{h0}(P^i_j)$ and its extension onto the whole domain $\omega$ by zero. Now we carry out the sweeping procedure below.

Given $w\in V^{\rm RQM}_{h0}(\mathcal{T}_\omega)$, we begin with the first row $\mathcal{T}^1$ of the subdivision. Note that $T^2_2$ is the only interior cell whose patch contain $T^1_1$, and there exists a unique $\phi^2_2\in\vwmh(P^2_2)$, such that $\phi^2_2=w_h$ in $T^1_1$. Denote $w^1_1:=w-\phi^2_2$, and $w^1_1\in \vwmh(\mathcal{T}_\omega\setminus \{T^1_1\})$. Then, there exists a unique $\phi^2_3\in\vwmh(P^2_3)$, such that $\phi^2_3=w^1_1$ on $T^1_2$. Further, $w^1_2:=w^1_1-\phi^2_3\in\vwmh(\mathcal{T}-\{T^1_1,T^1_2\})$. By repeating the procedure along the first row of the subdivision, we will obtain a $w^1_{n-2}=w-\sum_{j=2}^{n-1}\phi^2_j$ with $\phi^2_j\in\vwmh(P^2_j)$, and $w^1_{n-2}\in \vwmh(\mathcal{T}_\omega\setminus\cup_{j=1}^n-2\{T^1_j\})$. By the compatible condition, it follows that $w^1_{n-2}=0$ on $T^1_{n-1}\cup T^1_n$, and $w^1_{n-2}\in\vwmh(\mathcal{T}_h\setminus \mathcal{T}^1)$.

Repeat the procedure along the rows $\mathcal{T}^i$, $i=2,\dots,m-2$, we can represent $w=w^{m-2}_{n-2}+\sum_{i=1}^{m-2}\sum_{j=2}^{n-1}\phi^{i+1}_j$, with $w^{m-2}_{n-2}\in \vwmh(\mathcal{T}^{m-1}\cup\mathcal{T}^m)$, and $\phi^i_j\in\vwmh(P^i_j)$. By the virtue of Remark \ref{rem:noicnowm}, it is easy to verify $w^{m-2}_{n-2}=0$. Namely $w=\sum_{i=1}^{m-2}\sum_{j=2}^{n-1}\phi^{i+1}_j$ with $\phi^i_j\in\vwmh(P^i_j)$. This way, any function in $\vwmh$ can be represented as a linear combination of basis functions of $\vwmh(P^i_j)$, $i=2,\dots,m-1$, $j=2,\dots,n-1$, uniquely. The proof is completed.
\end{proof}

\begin{remark}
The set of the homogeneous RQM element function constructed on each $3\times3$ form a basis of $\vwmh(\mathcal{G}_h)$, which is used for programming. 
\end{remark}
\begin{remark}
If $\Omega$ is a domain such that all sides are parallel axes, it can be subdivided to rectangular blocks, and the sweeping procedure can be run row by row and block by block. 
\end{remark}

\subsection{General implementation on quadrilateral subdivisions}

When the subdivision consist of arbitrary quadrilaterals, there may be too many patterns and we do not seek to construct clearly a set of basis functions of $V^{\rm RQM}_{h0}$. We instead present an approach how \eqref{eq:modeldis} can be implemented. 

We begin with the fact that $V^{\rm RQM}_{h0}=\{w_h\in V^{\rm QW}_{h0}:\int_e\frac{\partial w_h}{\partial n}=0,\ \forall\,e\in \mathcal{E}_h\}.$ Define $\mathcal{P}_0(\mathcal{E}_h)$ the space of piecewise constant functions defined on $\mathcal{E}_h$. An equivalent formulation of \eqref{eq:modeldis} is then to find $(u_h,\lambda_h)\in V^{\rm QW}_{h0}\times \mathcal{P}_0(\mathcal{E}_h)$, such that 
\begin{equation}
\label{eq:aug}
\left\{
\begin{array}{llll}
\displaystyle(\nabla_h^2u_h,\nabla_h^2v_h)&\displaystyle+\sum_{e\in\mathcal{E}_h}\fint_e\llbracket \frac{\partial v}{\partial n}\rrbracket\lambda_h&=&(f,v_h)
\\
\displaystyle\sum_{e\in\mathcal{E}_h}\fint_e\llbracket \frac{\partial u_h}{\partial n}\rrbracket\mu_h&&=&0.
\end{array}
\right.
\end{equation}
The lemma below is direct. 
\begin{lemma}
The problem \eqref{eq:aug} admits a solution $(u_h,\lambda_h)$, and $u_h\in V^{\rm RQM}_{h0}$ solves \eqref{eq:modeldis}. Moreover, if $(u_h,\lambda_h)$ and $(\hat u_h,\hat \lambda_h)$ are two solutions of \eqref{eq:aug}, then $u_h=\hat u_h$.
\end{lemma}

\begin{remark}
Actually, it holds that
$V^{\rm RQM}_{h0}=V^{\rm QW}_{h0}\cap V^{\rm QM}_{h0}$. 
\end{remark}

\section{Concluding remarks}
\label{sec:conc}

In this paper, we present an approach how a minimal-degree consistent finite element space can be constructed for biharmonic equation on quadrilateral grids. The finite element space is designed, and a practical approach how it can be implemented is also given.  Technically, the two main ingredients for the analysis are the exact relation between the space and a vector Park-Sheen element space, and a generally unstable $P_1-P_0$ pair for Stokes problem. We remark that, on rectangular grids, the role of the $P_1-P_0$ pair may be replaced by $Q_1-P_0$ pair; we refer to \cite{Pitkaranta.J;Stenberg.R1985} for related discussion. 

In this space, we focus ourselves on the homogeneous Dirichlet boundary value problem of the biharmonic equation. The homogeneous Navier type boundry value problem can be studied in future. As RQM element space is a subspace of rectangular Morley element space, and the Morley element space with Navier type boundry value condition can be used for Poisson equation with high accuracy of $\mathcal{O}(h^2)$ order on uniform subdivisions, the RQM element space can be expected to be an optimal quadratic element for Poisson equation there on. This will be studied in future. Also, problems of higher orders can be studied as well. 

When the grid is rectangular, an explicit set of locally-supported basis functions are given. For other grids, the finite element scheme is implemented in an indirect way, namely, rewritten in the formulation of \eqref{eq:aug}. The sufficiency of \eqref{eq:aug} for the primal formulation is obtained. In the future, the well-posed-ness of \eqref{eq:aug} could be studied. The construction of an explicit set of locally-supported basis functions on general quadrilateral grids will also be discussed. 

Finally, the RQM finite element space is defined in a way similar to spline but with less smoothness. We can treat this as some nonconforming spline function. This paper is focused on quadrilateral subdivisions, and its generalization to less regular grids, such as triangular grids can be of theoretical and practical meaning. This will be studied in future.


\end{document}